\documentclass[a4paper]{jpconf}
\usepackage{graphicx}

\usepackage{amssymb}
\usepackage{amsmath}
\usepackage{amsthm}
\usepackage{amsfonts}

\newtheorem{theorem}{Theorem}[section]
\newtheorem{corollary}[theorem]{Corollary}
\newtheorem{lemma}[theorem]{Lemma}

\newtheorem{proposition}[theorem]{Proposition}

\newtheorem{remark}[theorem]{Remark}

\numberwithin{equation}{section}
\def\bn{{\mathbb N}}

\def\br{\mathbb{R}}

\def\m{\mu}

\def\cb{{\mathcal B}}
\def\ce{{\mathcal E}}

\def\bi{{\mathbb I}}

\def\bn{{\mathbb N}}

\def\bq{{\mathbb Q}}
\def\br{{\mathbb R}}

\def\bz{{\mathbb Z}}

\begin{document}
\title{The Dynamics of The Potts-Bethe Mapping over $\mathbb{Q}_p:$ The Case $p\equiv 2 \ (mod \ 3) $}

\author{Mansoor Saburov$^1$, Mohd Ali Khameini Ahmad$^1$}

\address{$^1$Department of Computational \& Theoretical Sciences, Faculty of Science, International Islamic University Malaysia, 
P.O. Box, 25200, Kuantan, Pahang, MALAYSIA
}

\ead{msaburov@iium.edu.my,  msaburov@gmail.com, khameini.ahmad@gmail.com}

\begin{abstract} 
In this paper, we study the dynamics of the Potts-Bethe mapping associated with the $p-$adic $q-$state Potts model over the Cayley tree of order three. Namely, we establish the regularity of the Potts-Bethe mapping for the case $p\equiv 2 \ (mod \ 3)$ with $p\geq5$.  	
\end{abstract}

\section{Introduction}

The effective method for investigation of phase transitions in statistical mechanics is to study the distribution of zeros of the grand partition function which is considered as a function of a complex magnetic field. In 1952, Lee and Yang, in their famous papers (see \cite{LeeYang,YangLee}), proved \textit{the circle theorem} which states that the zeros of the grand partition function of Ising ferromagnet lie on the unit circle in the complex magnetic plane (\textit{Lee--Yang zeros}). They proposed a mechanism for the occurrence of phase transitions in the thermodynamic limit and showed that the distribution of the zeros of a model determines its critical behavior. In 1964, Fisher, in his paper \cite{Fisher}, initiated the study of partition function zeros in the complex temperature plane (\textit{Fisher zeros}). These methods are then extended to other type of interactions and found a wide range of applications \cite{Arndt,AHR,IPZ1983,LuWu1998,LuWu2000,Monroe1991,PF1999,Ruelle}. The density of Lee--Yang zeros of the grand partition function of the Ising model can be determined experimentally from the field dependence of the isothermal magnetization data (see \cite{Binek}). The fractal structure of Fisher zeros in $q-$state Potts model (which is a multi-state generalization of the Ising model) on the diamond lattice was obtained by Derrida, de Seze and Itzykson \cite{DDI1983}. They showed that the Fisher zeros in hierarchical lattice models are just the Julia set corresponding to the renormalization group transformation. Bosco and Goulard Rosa \cite{BoscoRosa1987} (see also \cite{Monroe1995})  calculated the fractal dimension of the Julia set associated with the Lee--Yang zeros of the ferromagnetic Ising Model on the Cayley tree. By means of the dynamical system approach, Monroe, in series of papers \cite{Monroe1991,Monroe1992,Monroe1994,Monroe1996,Monroe2001}, studied the Julia set, the period doubling cascades, and the chaos behavior of the Potts-Bethe mapping associated with the $q-$state Potts model on the Bethe lattice (the Cayley tree).

The $p-$adic counterpart of the theory of Gibbs measures on Cayley trees has been also initiated in the papers \cite{GMR,MR1,MR2}. The existence of $p-$adic Gibbs measures as well as the phase transition for some $p-$adic lattice models were established in \cite{M2,M3,MH2013,FMMSOK2015,FMMSOK2016,RUKO2013}. Recently \cite{RUKO2015,SMAA2015c,SMAA2015d}, all translation-invariant $p-$adic Gibbs measures of the $p-$adic Potts model on the Cayley tree of order two and three were described by studying allocation of roots of quadratic and cubic equations over some domains of the $p-$adic fields. However, in the $p-$adic case, due to lack of the convex structure of the set of $p-$adic (quasi) Gibbs measures, it was quite difficult to constitute a phase transition with some features of the set of $p-$adic (quasi) Gibbs measures. The rigorous mathematical foundation of the theory of Gibbs measures on Cayley trees was presented in the book \cite{RUBook} (for survey see \cite{RUSurvey}).

Unlike the real case \cite{CKURRK}, the set of $p-$adic Gibbs measures of $p-$adic lattice models on the Cayley tree has a complex structure in a sense that it is strongly tied up with a Diophantine problem over the $p-$adic numbers fields. The rise of the order of the Cayley tree makes difficult to study the corresponding Diophantine problem over the $p-$adic numbers fields. In general, because of different topological structures, the same Diophantine problem may have different solutions from the field of $p-$adic numbers to the field of real numbers. The most frequently asked question is whether a root of a polynomial equation belongs to domains $\mathbb{Z}_p^{*}, \ \mathbb{Z}_p\setminus\mathbb{Z}_p^{*}, \ \mathbb{Z}_p, \ \mathbb{Q}_p\setminus\mathbb{Z}_p^{*}, \ \mathbb{Q}_p\setminus\left(\mathbb{Z}_p\setminus\mathbb{Z}_p^{*}\right), \ \mathbb{Q}_p\setminus\mathbb{Z}_p, \ \mathbb{Q}_p, \ \mathbb{S}_{p^{m}}(0)$ or not.  To the best of our knowledge, the little attention was given to this problem in the literature. Recently, this problem was studied for monomial equations \cite{FMMS}, quadratic equations \cite{SMAA2015c}, depressed cubic equations for primes $p>3$ in \cite{FMBOMSKM,FMBOMS,SMAA2016b} and for primes $p=2,3$ in \cite{SMAA2014,SMAA2015a,SMAA2015b} and bi-quadratic equations \cite{SMAA2016a}.

In this paper, we study the dynamics of the Potts-Bethe mapping associated with the $p-$adic $q-$state Potts model on the Cayley tree of order three. The paper is organized as follows. In Section 2, we first recall some basic concepts and notions in $p-$adic statistical mechanics. In Section 3, we show that the description of all translation invariant $p-$adic Gibbs measures can lead to the investigation of fixed points of the Potts-Bethe mapping in some domains of $\mathbb{Q}_p$. Finally, in Section 4, we show regularity of the Potts-Bethe mapping for the case $p\equiv 2 \ (mod \ 3)$. The case $p\equiv 1 \ (mod \ 3) $ would be studied elsewhere. Throughout this paper, we always assume that $p\geq5$ unless otherwise mentioned.

\section{Preliminaries}

\subsection{The $p-$adic numbers}

For a fixed prime $p$, the field $\mathbb{Q}_p$ of $p-$adic
numbers is a completion of the rational numbers $\mathbb{Q}$ with
respect to the non-Archimedean norm $|\cdot|_p:\mathbb{Q}\to\br$ given by
\begin{eqnarray*}
|x|_p=\left\{
\begin{array}{c}
p^{-k} \ x\neq 0,\\
0,\ \quad x=0,
\end{array}
\right.
\end{eqnarray*}
where $x=p^k\frac{m}{n}$ with $k,m\in\mathbb{Z},$ $n\in\bn$,
$(m,p)=(n,p)=1$. The number $k$ is called \textit{a $p-$order} of $x$
and it is denoted by $ord_p(x)=k.$

Any $p-$adic number $x\in\mathbb{Q}_p$ can be uniquely represented in the
canonical form
$
x=p^{ord_p(x)}\left(x_0+x_1\cdot p +x_2\cdot p^2+\cdots \right)
$
where $x_0\in \{1,2,\cdots p-1\}$ and $x_i \in \{ 0,1,2,\cdots p-1 \}$ for
$i\in \mathbb{N}$. We respectively denote the set of all {\it $p-$adic integers} and {\it $p-$adic units} of
$\mathbb{Q}_p$ by
$\mathbb{Z}_p=\{x\in\mathbb{Q}_{p}: |x|_p\leq1\}$ and  $\mathbb{Z}_p^{*}=\{x\in\bq_{p}: |x|_p=1\}.$
Any $p-$adic unit $x\in\bz_p^{*}$ has the unique canonical form
$x=x_0+x_1\cdot p+x_2\cdot p^2+\cdots$
where $x_0\in \{1,2,\cdots p-1\}$ and $x_i\in\{0,1,2,\cdots p-1\}$ for $i\in\bn.$
Any nonzero $p-$adic number $x\in\bq_p$ has the unique representation $x =\frac{x^{*}}{|x|_p}$, where $x^{*}\in\bz_p^{*}$ (see \cite{BorShaf,NK}). We set $\mathbb{B}(a,r)=\{x\in \bq_p : |x-a|_p< r\}$ for $a\in \bq_p$ and $r>0$. 

The $p$-adic logarithm $\log_p(\cdot) :\mathbb{B}(1,1)\to \mathbb{B}(0,1)$ is defined as follows
$$
\log_p(x)=\log_p(1+(x-1))=\sum_{n=1}^{\infty}(-1)^{n+1}\frac{(x-1)^n}{n}.
$$
The $p$-adic exponential $\exp_p(\cdot): \mathbb{B}(0,p^{-1/(p-1)})\to \mathbb{B}(1,1)$ is defined as follows
$$
\exp_p(x)=\sum_{n=0}^{\infty}\frac{x^n}{n!}.
$$

\begin{lemma}[\cite{NK}]\label{21}  Let $x\in
	\mathbb{B}(0,p^{-1/(p-1)}).$ Then we have that  $$ |\exp_p(x)|_p=1,\ \ 
	|\exp_p(x)-1|_p=|x|_p<1, \ \ |\log_p(1+x)|_p=|x|_p<p^{-1/(p-1)},$$
	$$ \log_p(\exp_p(x))=x, \ \ \exp_p(\log_p(1+x))=1+x. $$
\end{lemma}

Let $\ce_p=\{ x\in\mathbb{Q}_p:   |x-1|_p<p^{-1/(p-1)}\}.$ Obviously, $\ce_p=\{ x \in \mathbb{Z}_p^{*} : |x-1|_p<1\}=\mathbb{B}(1,1)$ whenever $p\geq3.$ 

\subsection{The $p-$adic measure}

Let $(X,\cb)$ be a measurable space, where $\cb$ is an algebra of
subsets $X$. A function $\m:\cb\to \mathbb{Q}_p$ is said to be a {\it
	$p$-adic measure} if for any $A_1,\dots,A_n\subset\cb$ such that
$A_i\cap A_j=\emptyset$ ($i\neq j$) the equality holds
$$
\mu\bigg(\bigcup_{j=1}^{n} A_j\bigg)=\sum_{j=1}^{n}\mu(A_j).
$$

A $p-$adic measure is called a {\it probability measure} if
$\mu(X)=1$.  A $p-$adic
probability measure $\m$ is called {\it bounded} if
$\sup\{|\m(A)|_p : A\in \cb\}<\infty $. The boundedness condition of $p-$adic measures provides an integration theory. The reader may refer to \cite{Kh1994,Kh2009,Kh1999} for more detailed information about $p-$adic measures.

\subsection{The Cayley tree}

Let $\Gamma^k_+ = (V,L)$ be a semi-infinite Cayley tree of order
$k\geq 1$ with the root $x^0$ (whose each vertex has exactly $k+1$
edges except for the root $x^0$ which has $k$ edges). Let $V$ be
a set of vertices and $L$ be a set of edges. The vertices $x$
and $y$ are called {\it nearest neighbors}, denoted by
$l=\langle{x,y}\rangle,$ if there exists an edge $l\in L$ connecting them. A collection of
the pairs $\langle{x,x_1}\rangle,\cdots,\langle{x_{d-1},y}\rangle$ is called {\it a path} between vertices $x$ and $y$. The distance $d(x,y)$ between $x,y\in V$ on
the Cayley tree is the length of the shortest path between $x$ and $y$. Let 
$$
W_{n}=\left\{ x\in V: d(x,x^{0})=n\right\}, \quad
V_n=\overset{n}{\underset{m=1}{\bigcup}}W_{m}, \quad L_{n}=\left\{
<x,y>\in L: x,y\in V_{n}\right\}.
$$
The set of direct successors of $x$ is defined by
$$
S(x)=\left\{ y\in W_{n+1}:d(x,y)=1\right\}, \ \ \forall\  x\in W_{n}.
$$

We now introduce a coordinate structure in $V$.
Every vertex $x\neq x^{0}$ has the coordinate
$(i_1,\cdots,i_n)$ where $i_m\in\{1,\dots,k\},\ 1\leq m\leq n$ and
the vertex $x^0$ has the coordinate $(\emptyset)$. More precisely, the symbol
$(\emptyset)$ constitutes level $0$ and the sites $i_1,\dots,i_n$ form
level $n$ of the lattice. In this case, for any $x\in V,\
x=(i_1,\cdots,i_n)$ we have that
$$
S(x)=\{(x,i): 1\leq i\leq k\},
$$
here $(x,i)$ means that $(i_1,\cdots,i_n,i)$. Let us define a binary operation $\circ:V\times V\to V$
as follows: for any two elements $x=(i_1,\cdots,i_n)$ and $y=(j_1,\cdots,j_m)$ we define
$$
x\circ y=(i_1,\cdots,i_n)\circ(j_1,\cdots,j_m)=(i_1,\cdots,i_n,j_1,\cdots,j_m)
$$
and
$$
y\circ x=(j_1,\cdots,j_m)\circ(i_1,\cdots,i_n)=(j_1,\cdots,j_m,i_1,\cdots,i_n).
$$
Then  $(V,\circ)$ is a noncommutative semigroup with a unit $x^0$.
Now, we can define a translation $\tau_g: V\to V$ for $ g\in V$
as follows $\tau_g(x)=g\circ x.$
Consequently, by means of $\tau_g,$ we define a translation $\tilde\tau_g:L\to L$ as $
\tilde\tau_g(\langle{x,y}\rangle)=\langle{\tau_g(x),\tau_g(y)}\rangle.$

Let $G\subset V$ be a sub-semigroup of $V$ and $h:L\to\mathbb{Q}_p$ be a function.
We say that $h$ is $G-${\it periodic} if $h(\tilde\tau_g(l))=h(l)$ for all $g\in G$ and $l\in L$.
Any $V-$periodic function is called {\it translation-invariant}.

\subsection{The $p-$adic Potts model}

Let $\Phi=\{1,2,\cdots, q\}$ be a finite set.
A configuration (resp. a finite volume configuration, a boundary configuration) is a function $\sigma:V\to \Phi$ (resp. $\sigma_n:V_n\to \Phi,$ $\sigma^{(n)}:W_n\to \Phi$). We denote by $\Omega$ (resp. $\Omega_{V_n}$, $\Omega_{W_n}$) a set of all configurations (resp. all finite volume configurations, all boundary configurations). For given configurations $\sigma_{n-1}\in\Omega_{V_{n-1}}$ and $\sigma^{(n)}\in\Omega_{W_n}$, we define their concatenation to be a finite volume configuration $\sigma_{n-1}\vee\sigma^{(n)}\in\Omega_{V_n}$  such that
$$
\left(\sigma_{n-1}\vee\sigma^{(n)}\right)(v)=
\begin{cases}
\sigma_{n-1}(v) & \text{if} \ \ v\in V_{n-1} \\
\sigma^{(n)}(v) & \text{if} \ \ v\in W_n
\end{cases}.
$$

The Hamiltonian of \textit{a $p-$adic Potts model} with the spin value set $\Phi=\{1,2,\cdots q\}$ on the finite volume configuration is defined as follows
\begin{eqnarray*}\label{Hamiltonian}
H_n(\sigma_n)=J\sum\limits_{\langle x,y\rangle\in L_n}\delta_{\sigma_n(x)\sigma_n(y)},
\end{eqnarray*}
for all $\sigma_n\in \Omega_{V_n}$ and $n\in\mathbb{N}$ where $J\in \mathbb{B}(0, p^{-1/(p-1)})$ is a coupling constant, $\langle x,y\rangle$ stands for nearest neighbor vertices, and $\delta$ is Kronecker's symbol. 

\subsection{The $p-$adic Gibbs measure}

Let us present a construction of a $p-$adic Gibbs measure of the $p-$adic Potts model with $q-$spin values. We define a $p-$adic measure $\mu^{(n)}_{\mathbf{\tilde{h}}}:\Omega_{V_n}\to\mathbb{Q}_p$ associated with a boundary function $\mathbf{\tilde{h}}:V\to\mathbb{Q}_p^{q},$ $\mathbf{\tilde{h}}(x)=\left({\tilde{h}}^{(1)}_{x},\dots, {\tilde{h}}^{(q)}_{x}\right),$ $x\in V$ as follows
\begin{eqnarray}\label{mu_h^n}
\mu_{\mathbf{\tilde{h}}}^{(n)}(\sigma_n)
=\frac{1}{\mathcal{Z}^{(n)}_{\mathbf{\tilde{h}}}}\exp_p\left\{H_n(\sigma_n)+\sum_{x\in W_n}{\tilde{h}}^{(\sigma_n(x))}_{x}\right\}
\end{eqnarray}
for all $\sigma_n\in \Omega_{V_n}, \ n\in\mathbb{N}$ where $\exp_p(\cdot):\mathbb{B}(0, p^{-1/(p-1)})\to \mathbb{B}(1, 1)$ is the $p-$adic exponential function and $\mathcal{Z}^{(n)}_{\mathbf{\tilde{h}}}$ is {\it a partition function} defined by
\begin{eqnarray*}\label{ZN}
	\mathcal{Z}^{(n)}_{\mathbf{\tilde{h}}}=\sum_{\sigma_n\in\Omega_{V_n}}\exp_p\left\{H_n(\sigma_n)+\sum_{x\in W_n}{\tilde{h}}^{(\sigma_n(x))}_{x}\right\}
\end{eqnarray*}
for all $n\in\mathbb{N}.$ We always assume that $\mathbf{\tilde{h}}(x)\in\mathbb{B}(0, p^{-1/(p-1)})^{\times q}$ for any $x\in V.$

The $p-$adic measures \eqref{mu_h^n} are called \textit{compatible} if one has that
\begin{equation}\label{compatibility}
\sum_{\sigma^{(n)}\in \Omega_{W_n}}\mu_{\mathbf{\tilde{h}}}^{(n)}(\sigma_{n-1}\vee \sigma^{(n)})=\mu_{\mathbf{\tilde{h}}}^{(n-1)}(\sigma_{n-1})
\end{equation}
for all $\sigma_{n-1}\in \Omega_{V_{n-1}}$ and $n\in \mathbb{N}$.

Due to the Kolmogorov extension theorem of the $p-$adic measures \eqref{mu_h^n} (see \cite{GMR,Kh2009}), there exists a unique $p-$adic measure $\mu_{\mathbf{\tilde{h}}}:\Omega\to\mathbb{Q}_p$ such that
$$\mu_{\mathbf{\tilde{h}}}(\{\sigma\mid_{V_n}=\sigma_n\})=\mu_{\mathbf{\tilde{h}}}^{(n)}(\sigma_n), \quad \forall \ \sigma_n\in \Omega_{V_n}, \ \forall\  n\in\mathbb{N}.$$ 
This uniquely extended measure $\mu_{\mathbf{\tilde{h}}}:\Omega\to\mathbb{Q}_p$ is called \textit{a $p-$adic Gibbs measure}. The following theorem describes the condition on the boundary function $\mathbf{{\tilde{h}}}:V\to\mathbb{Q}_p^{q}$ in which the compatibility condition \eqref{compatibility}  is satisfied.

\begin{theorem}[\cite{MR1,MR2}]\label{Equationforh} Let $\mathbf{{\tilde{h}}}:V\to\mathbb{Q}_p^{q},$ $\mathbf{\tilde{h}}(x)=\left(\tilde{h}^{(1)}_{x},\cdots, \tilde{h}^{(q)}_{x}\right)$ be a given boundary function and $\mathbf{{{h}}}:V\to\mathbb{Q}_p^{q-1},$ $\mathbf{h}(x)=\left(h^{(1)}_{x},\cdots, h^{(q-1)}_{x}\right)$ be a function defined as $h^{(i)}_{x}={\tilde{h}}^{(i)}_{x}-{\tilde{h}}^{(q)}_{x}$ for all $ i=\overline{1,q-1}.$ Then the $p-$adic probability distributions $\left\{\mu_{\mathbf{\tilde{h}}}^{(n)}\right\}_{n\in\mathbb{N}}$  are compatible if and only if one has that
	\begin{equation}\label{systemofequationforh}
	\mathbf{{h}}(x)=\sum_{y\in S(x)}\mathbf{F}(\mathbf{h}(y)), \quad \forall \ x\in V\setminus\{x^0\},
	\end{equation}
	where $S(x)$ is the set of direct successors of $x$ and the function $\mathbf{F}:\mathbb{Q}_p^{q-1}\to \mathbb{Q}_p^{q-1},$ $\mathbf{F}(\mathbf{h})=(F_1,\cdots,F_{q-1})$ for $ \mathbf{h}=(h_1, \cdots,h_{q-1})$ is defined as follows
	$$F_i=\log_p\left({(\theta-1)\exp_p(h_i)+\sum_{j=1}^{q-1}\exp_p(h_j)+1\over \theta+ \sum_{j=1}^{q-1}\exp_p(h_j)}\right), \quad \theta=\exp_p(J).$$
\end{theorem}

\subsection{The translation-invariant $p-$adic Gibbs measure} A $p-$adic Gibbs measure is translation-invariant if and only if the boundary function $\mathbf{{\tilde{h}}}:V\to\mathbb{Q}_p^{q}$ is constant, i.e., $\mathbf{{\tilde{h}}}(x)=\mathbf{\tilde{h}}$ for any $x\in V.$ In this case, the condition \eqref{systemofequationforh} takes the form $\mathbf{h}=k\mathbf{F}(\mathbf{h})$ or equivalently
$$
h_i=\log_p\left({(\theta-1)\exp_p(h_i)+\sum_{j=1}^{q-1}\exp_p(h_j)+1\over \theta+ \sum_{j=1}^{q-1}\exp_p(h_j)}\right)^k, \quad \forall \ i=\overline{1,q-1}.
$$ 
Let $\mathbf{z}=(z_1,\cdots,z_{q-1})\in\mathcal{E}_p^{q-1}$ such that $z_i=\exp_p(h_i)$ for any $i=\overline{1,q-1}.$ In this case, we write $\mathbf{z}=\exp_p(\mathbf{h}).$ Hence, we obtain from the last system of equations
\begin{eqnarray*}
z_i=\left({(\theta-1)z_i+\sum_{j=1}^{q-1}z_j+1\over \theta+ \sum_{j=1}^{q-1}z_j}\right)^k, \quad \forall \ i=\overline{1,q-1}.
\end{eqnarray*}

Consequently, we have the following result.

\begin{theorem}[\cite{MR1,MR2}]\label{existenceTIpGM}
	There exists a TIpGM  $\mu_{\mathbf{\tilde{h}}}:\Omega\to\mathbb{Q}_p$ associated with a boundary function $\mathbf{{\tilde{h}}}:V\to\mathbb{Q}_p^{q},$ $\mathbf{\tilde{h}}(x)=\mathbf{\tilde{h}}=(\tilde{h}_1,\cdots, \tilde{h}_{q})$ for any $x\in V$ if and only if $\mathbf{z}=\exp_p(\mathbf{h})\in\mathcal{E}_p^{q-1}$ is a solution of the following system of equations 
	\begin{equation}\label{equationwrtz}
	z_i=\left({(\theta-1)z_i+\sum_{j=1}^{q-1}z_j+1\over \theta+ \sum_{j=1}^{q-1}z_j}\right)^k,\quad \ i=\overline{1,q-1}.
	\end{equation}
	where $\mathbf{z}=(z_1,\cdots,z_{q-1}), \mathbf{h}=(h_1,\cdots, h_{q-1})$ such that $h_i=\tilde{h}_i-\tilde{h}_q, \forall i=\overline{1,q-1}.$
\end{theorem}

\begin{remark}
	If $\theta=1$ then the system of equation \eqref{equationwrtz} has a unique solution $\mathbf{z}=(1,1,\cdots,1).$ In what follows, we always assume that $\theta\neq 1$.
\end{remark}

In the cases $k=2,3$ the description of all TIpGMs were given in the papers \cite{RUKO2015,SMAA2015d}.

\section{Translation Invariant $p-$adic Gibbs Measures}

The following theorem describes all TIpGMs of the Potts model with $q$ spin values on the Cayley tree of order three.

\begin{theorem}[Descriptions of TIpGMs, \cite{SMAA2015d}]\label{Descriptionq>3}
There exists a TIpGM $\mu_{\mathbf{\tilde{h}}}:\Omega\to\mathbb{Q}_p$ associated with a boundary function  $\mathbf{\tilde{h}}=(\tilde{h}_1,\cdots, \tilde{h}_{q})$ if and only if $\tilde{h}_j=\log_p(hz_j)$ for all $j=\overline{1,q-1}$ and $\tilde{h}_q=\log_p(h)$ where $h\in\mathcal{E}_p$ is any $p-$adic number and $\mathbf{z}=(z_1,\cdots z_{q-1})\in \mathcal{E}_p^{q-1}$ is defined either one of the following form
\begin{itemize}
	\item[(\textbf{A})] $\mathbf{z}=\left(1,\cdots,1\right);$
	\item[(\textbf{B})] $\mathbf{z}=\left(z,\cdots,z\right)$ where $z\in\mathcal{E}_p\setminus\{1\}$ is a root of the following cubic equation 
	\begin{multline*}
	\quad \quad \quad (q-1)^3z^3+\left(3(q-1)^2-(\theta-1)^2(\theta+3(q-1)-1)\right)z^2\\
	+\left(3(q-1)-(\theta-1)^2(\theta+2)\right)z+1=0;
	\end{multline*}
	\item[(\textbf{C})] $\mathbf{z}=\mathbf{e}_{\alpha_1}+z\mathbf{e}_{\alpha_2}$ with $|\alpha_i|=m_i, \ m_1+m_2=q-1$ such that  $z\in\mathcal{E}_p\setminus\{1\}$ is a root of the following cubic equation
	\begin{multline*}
	\quad \quad \quad	m_2^3z^3+\left[3m_2^2(m_1+1)-(\theta-1)^2(\theta+3m_2-1)\right]z^2\\
	+\left[3m_2(m_1+1)^2-(\theta-1)^2(\theta+3(m_1+1)-1)\right]z+(m_1+1)^3=0;
	\end{multline*}
	\item[(\textbf{D})] $\mathbf{z}=z_1\mathbf{e}_{\alpha_1}+z_2\mathbf{e}_{\alpha_2}$  with $|\alpha_i|=m_i, \ m_1+m_2=q-1$ such that  
	\begin{itemize}
		\item[(i)] 	If  $m_1\neq \frac{1-\theta}{3}$ then 
		$z_1=-\frac{(\theta-1+3m_2)z_2+\theta+2}{(\theta-1+3m_1)}\in\mathcal{E}_p\setminus\{1\}$
		and $z_2\in\mathcal{E}_p\setminus\{1\}$ is a root of the following cubic equation 
		\begin{multline*}
		\quad \quad \quad\quad \quad \left[(m_1-m_2)z+(m_1-1)\right]^3\\
		+(\theta-1+3m_1)^2\left[(\theta-1+3m_2)z^2+(\theta+2)z\right]=0;
		\end{multline*}
		\item[(ii)] If  $m_2\neq \frac{1-\theta}{3}$ then 
		$z_2=-\frac{(\theta-1+3m_1)z_1+\theta+2}{(\theta-1+3m_2)}\in\mathcal{E}_p\setminus\{1\}$
		and $z_1\in\mathcal{E}_p\setminus\{1\}$ is a root of the following cubic equation 
		\begin{multline*}
		\quad \quad \quad\quad \quad \left[(m_2-m_1)z+(m_2-1)\right]^3\\
		+(\theta-1+3m_2)^2\left[(\theta-1+3m_1)z^2+(\theta+2)z\right]=0;
		\end{multline*}
	\end{itemize}
	\item[(\textbf{E})] $\mathbf{z}=z_1\mathbf{e}_{\alpha_1}+z_2\mathbf{e}_{\alpha_2}+\mathbf{e}_{\alpha_3}$ with $|\alpha_i|=m_i, \ m_1+m_2+m_3=q-1$ such that 
	\begin{itemize}
		\item[(i)] If $m_1\neq \frac{1-\theta}{3}$ then 
		$z_1=-\frac{(\theta-1+3m_2)z_2+3m_3+\theta+2}{(\theta-1+3m_1)}\in\mathcal{E}_p\setminus\{1\}$
		and $z_2\in\mathcal{E}_p\setminus\{1\}$ is a root of the following cubic equation
		\begin{multline*}
		\quad \quad \quad\quad \quad \left[(m_1-m_2)z+(m_1-m_3-1)\right]^3+(\theta-1+ 3m_1)^2(\theta-1+3m_2)z^2\\
		+(\theta-1+3m_1)^2(3m_3+\theta+2)z=0;
		\end{multline*}
		\item[(ii)]  If  $m_2\neq \frac{1-\theta}{3}$ then 
		$z_2=-\frac{(\theta-1+3m_1)z_1+3m_3+\theta+2}{(\theta-1+3m_2)}\in\mathcal{E}_p\setminus\{1\}$
		and $z_1\in\mathcal{E}_p\setminus\{1\}$ is a root of the following cubic equation 
		\begin{multline*}
		\quad \quad \quad\quad \quad	\left[(m_2-m_1)z+(m_2-m_3-1)\right]^3+(\theta-1+3m_2)^2(\theta-1+3m_1)z^2\\
		+(\theta-1+3m_2)^2(3m_3+\theta+2)z=0;	
		\end{multline*}
		\item[(iii)] If $\theta=1-q$ and $m_1=m_2=m_3+1$ then either $z_1\in\mathcal{E}_p\setminus\{1\}$ or $z_2\in\mathcal{E}_p\setminus\{1\}$ is any $p-$adic number  so that the second one is a root of the cubic equation $(z_1+z_2+1)^3=27z_1z_2$.
	\end{itemize}
\end{itemize}
\end{theorem}

\section{Dynamics of the Potts--Bethe mapping}

It follows from Theorems \ref{existenceTIpGM} and \ref{Descriptionq>3}, in order to find TIpGM, we have to find all fixed points of \textit{the Potts--Bethe mapping} $f_{\theta, q, k} : \mathbb{Q}_p \to \mathbb{Q}_p$ defined as
\begin{equation}\label{IsingPottsGeneral}
	f_{\theta, q, k}(x)=\left(\frac{\theta x+q-1}{x+\theta+q-2}\right)^k
\end{equation}
where $\theta\in\mathcal{E}_p$ and $q,k\in\mathbb{N}$ such that $\theta\neq 1,\ \theta\neq 1-q.$ Therefore, we are aiming to study the dynamics of the Potts--Bethe mapping. It is well defined on the set $\textup{\textbf{Dom}}\{f_{\theta,q,k}\}:=\mathbb{Q}_p\setminus \{\mathbf{x}^{(\infty)}\}$ where $\mathbf{x}^{(\infty)}:=2-\theta-q$.

Throughout this paper, we always assume that $|\theta-1|_p<1, \ |q|_p<1$. This assumption is necessary to have non-unique translation invariant $p-$adic Gibbs measures.  The dynamics of the Potts--Bethe mapping for the cases $k=1$ and $k=2$ were studied in \cite{Liao2} and \cite{FMOK}, respectively. In this paper, we study the case $k=3$ with $p\equiv 2 \ (mod \ 3)$. The case $p\equiv 1 \ (mod \ 3) $ would be studied elsewhere. 

Let $f_{\theta, q, 3} : \textup{\textbf{Dom}}\{f_{\theta,q,3}\} \to \mathbb{Q}_p$ be the Potts--Bethe mapping 
\begin{equation}\label{IsingPottsk=3}
f_{\theta, q, 3}(x)=\left(\frac{\theta x+q-1}{x+\theta+q-2}\right)^3
\end{equation}
where $\textup{\textbf{Dom}}\{f_{\theta,q,3}\}:=\mathbb{Q}_p\setminus \{\mathbf{x}^{(\infty)}\}$ and $\mathbf{x}^{(\infty)}:=2-\theta-q$.

\subsection{The Fixed Point Set}
Let $\mathbf{Fix}\{f_{\theta, q, 3}\}:=\{x \in \mathbb{Q}_p : f_{\theta, q, 3}(x)=x \}$ be a set of all fixed points of the Potts--Bethe mapping. It is clear that $\mathbf{x}^{(0)}:=1 \in \mathbf{Fix}\{f_{\theta, q, 3}\}$. Moreover, it follows from $f_{\theta, q, 3}(x)-1=x-1$ that
\begin{multline*}
(x-1)(\theta-1)\frac{(\theta x+q-1)^2+(\theta x+q-1)(x+\theta+q-2)+(x+\theta+q-2)^2}{(x+\theta+q-2)^3}\\
=(x-1).	
\end{multline*}
Therefore, any other fixed point $x \neq \mathbf{x}^{(0)}$ is a root of the following cubic equation
\begin{multline*}
(\theta-1)\left[(\theta x+q-1)^2+(\theta x+q-1)(x+\theta+q-2)+(x+\theta+q-2)^2\right] \\
=(x+\theta+q-2)^3
\end{multline*}
or equivalently
\begin{equation}\label{wrtx}
(\theta-1)(\theta x+q-1)\left((\theta+1)x+\theta+2q-3\right)= (x+q-1)(x+\theta+q-2)^2.
\end{equation}

We introduce a new variable $y:=\frac{x-1+q}{\theta-1}+1$. The cubic equation \eqref{wrtx} can be written with respect to $y$ as follows
\begin{equation}\label{wrty}
y^3-(1+\theta+\theta^2)y^2-(2\theta+1)(1-\theta-q)y-(1-\theta-q)^2=0.
\end{equation}

Let us find all possible roots of the cubic equation \eqref{wrty} whenever $\theta \neq 1,\  \theta \neq 1-q$.

\begin{proposition}\label{structureroot}
Let $p\geq 5$ with $ p \equiv 2 \ (mod \ 3)$  and $|\theta-1|_p<1,\ |q|_p<1, \ (\theta-1)(1-\theta-q)\neq 0$. Then the cubic equation \eqref{wrty} has a unique root $\mathbf{y}^{(1)}=3+y^{(1)}_rp^r+\cdots$ in the $p-$adic field $\mathbb{Q}_p$. Moreover, we have that
\begin{eqnarray*}
\left|\mathbf{y}^{(1)}\right|_p=1\quad  \textup{and} \quad
\left|\mathbf{y}^{(1)}-3\right|_p=|q(4\theta+5)-(\theta-1)(4\theta+14)-q^2|_p.
\end{eqnarray*} 
\end{proposition}

\begin{proof} We first show that if the cubic equation \eqref{wrty} has any root in the $p-$adic field $\mathbb{Q}_p$ then it must lies in the set $\mathbb{Z}_p.$ Suppose the contrary, i.e. the cubic equation \eqref{wrty} has a root $y$ such that $|y|_p>1.$ Since $p\geq 5$ and $y^3=(1+\theta+\theta^2)y^2+(2\theta+1)(1-\theta-q)y+(1-\theta-q)^2,$ we obtain that 
$$|(1+\theta+\theta^2)y^2+(2\theta+1)(1-\theta-q)y+(1-\theta-q)^2|_p=|y|^2_p=|y|^3_p.$$ It is a contradiction. Therefore, any root of the cubic equation must lie in the set $\mathbb{Z}_p.$ We refer to \cite{SMAA2015d} for the detailed study of the general cubic equation over $\mathbb{Z}_p$.  

Let $a=-(1+\theta+\theta^2),\ b=-(2\theta+1)(1-\theta-q)$ and $c=-(1-\theta-q)^2$. One can easily verify that $|b|_p^2=|c|_p=|1-\theta-q|_p^2<1=|a|_p$. Let $h(y)=y^3+ay^2+by+c.$ Then $h'(y)=3y^2+2ay.$

Any root $\mathbf{y} = y_0+y_1p+\cdots$ of the cubic equation \eqref{wrty} which belongs to the set $\mathbb{Z}_p^*$ must satisfy the following condition $y_0^3-3y_0^2 \equiv 0 \ (mod \ p)$ or equivalently $\mathbf{y}\equiv y_0\equiv 3 \ (mod \ p)$. On the other hand, since $h(3)\equiv 0 \ (mod \ p)$ and  $h'(3)\equiv 9 \not\equiv 0\ (mod \ p)$ for $p\geq 5,$ due to Hensel's lemma the cubic equation \eqref{wrty} has a unique root over $\mathbb{Z}_p^*$ which satisfies  $\mathbf{y}\equiv 3 \ (mod \ p).$ Consequently, the cubic equation \eqref{wrty} has a unique root $\mathbf{y}^{(1)}=3+y_r^{(1)}p^r+\cdots$ over $\mathbb{Z}_p^*$ where $r \in\mathbb{N}.$

We now want to show that the cubic equation \eqref{wrty} does not have any root over $\mathbb{Z}_p\setminus \mathbb{Z}_p^{*}.$ Let $\delta_1=b^2-4ac$. It is clear that $\delta_1=-3(1-\theta-q)^2$ and $|b|_p^2=|a|_p|c|_p=|1-\theta-q|_p^2=|\delta_1|_p$. Since $p \equiv 2 \ (mod \ 3)$ and $-3$ is a \underline{quadratic  non-residue} modulo $p$, there \underline{does not exist} $\sqrt{\delta_1}.$ Therefore, the cubic equation \eqref{wrty} does not have any roots over $\mathbb{Z}_p\setminus \mathbb{Z}_p^{*}$ (see Theorem 5.1 (\textbf{A})(\textit{iv}), \cite{SMAA2015d}). Consequently, the cubic equation \eqref{wrty} has a unique root $\mathbf{y}^{(1)}=3+y^{(1)}_rp^r+\cdots$  over $\mathbb{Q}_p.$

 Lastly, in order to prove $\left|\mathbf{y}^{(1)}-3\right|_p=|q(4\theta+5)-(\theta-1)(4\theta+14)-q^2|_p$, we have to introduce a new variable $z:=y-3$. In this case, the cubic equation \eqref{wrty} takes the following form with respect to $z$
\begin{multline}\label{wrtz}
z^3+\left[8-\theta-\theta^2\right]z^2+\left[21-6\theta(\theta+1)-(2\theta+1)(1-\theta-q)\right]z \\ +\left[q(4\theta+5)-(\theta-1)(4\theta+14)-q^2\right]=0
\end{multline}
Let $A=(8-\theta-\theta^2),\ B=21-6\theta(\theta+1)-(2\theta+1)(1-\theta-q)$ and $C=q(4\theta+5)-(\theta-1)(4\theta+14)-q^2$. One can check that $|A|_p=1,\ |B|_p=1$ and $|C|_p=|q(4\theta+5)-(\theta-1)(4\theta+14)-q^2|_p \leq \max\{|q|_p,|\theta-1|_p\}<1$. Since $|B|_p=|A|_p^2$ and $|C|_p<|A|_p^3$, the cubic equation \eqref{wrtz} has a unique root $\mathbf{z}$ (see Theorem 5.1 \textbf{F}, \cite{SMAA2015d}. Indeed, we already proved it with respect to the variable $y$ above) such $|\mathbf{z}|_p=\frac{|C|_p}{|B|_p}=|q(4\theta+5)-(\theta-1)(4\theta+14)-q^2|_p$ (see the proof of Theorem 5.1,  \cite{SMAA2015d}). Hence, $|\mathbf{z}^{(1)}|_p=\left|\mathbf{y}^{(1)}-3\right|_p=|q(4\theta+5)-(\theta-1)(4\theta+14)-q^2|_p$. It completes the proof.
\end{proof}
\begin{corollary}\label{structurey1}
Let $p\geq 5$ with $ p \equiv 2 \ (mod \ 3)$  and $|\theta-1|_p<|q|_p<1, \ (\theta-1)(1-\theta-q)\neq 0, \ q=p^ms$ for some $s,m \in \mathbb{N},\ (s,p)=1$. Let $1 \leq s_0 \leq p-1$ such that $s_0 \equiv -s \ (mod \ p)$. Then the root $\mathbf{y}^{(1)}$ of the cubic equation \eqref{wrty} satisfies the following congruence
$$\mathbf{y}^{(1)} \equiv 3 + s_0p^m \ (mod \ p^{m+1}) \quad \textup{or equivalently} \quad |\mathbf{y}^{(1)}-3+q|_p<|\mathbf{y}^{(1)}-3|_p=|q|_p.$$
\end{corollary}
\begin{proof}
Due to Proposition \ref{structureroot}, we have that $|\mathbf{y}^{(1)}-3|_p=|q(4\theta+5)-(\theta-1)(4\theta+14)-q^2|_p=|q|_p$. It means that $\mathbf{y}^{(1)}=3+y^{(1)}_mp^m+\cdots$. We can rewrite the cubic equation \eqref{wrty} as follows
\begin{equation}\label{wrty2}
y^2(y-3)-(\theta+2)(\theta-1)y^2-(2\theta+1)(1-\theta-q)y-(1-\theta-q)^2=0.
\end{equation}
Since $|\theta-1|_p<|q|_p<1,$ we have that 
\begin{eqnarray*}
\left(\mathbf{y}^{(1)}\right)^2\equiv 9+6y^{(1)}_mp^m \ (mod \ p^{m+1}), \qquad
\theta\equiv  1 \ (mod \ p^{m+1}),\\
|(\theta+2)(\theta-1)\left(\mathbf{y}^{(1)}\right)^2+(2\theta+1)(1-\theta)\mathbf{y}^{(1)}+(1-\theta-q)^2|_p<|q|_p.
\end{eqnarray*}
It follows from  \eqref{wrty2} that
\begin{eqnarray*}
9y^{(1)}_mp^m+6\left(y^{(1)}_m\right)^2p^{2m}+9p^ms \equiv 0 \ (mod \ p^{m+1}) \quad \textup{equivalently} \quad y^{(1)}_m+s \equiv 0 \ (mod \ p).
\end{eqnarray*}
Therefore, we get $y^{(1)}_m=s_0 \equiv -s \ (mod \ p)$ which completes the proof.
\end{proof}

\begin{theorem}\label{fixedpoints}
Let $p\geq 5$ with $ p \equiv 2 \ (mod \ 3)$  and $|\theta-1|_p<1,\ |q|_p<1, \ (\theta-1)(1-\theta-q)\neq 0$. Let $\mathbf{y}^{(1)}$ be a unique root of the cubic equation \eqref{wrty} and  $\mathbf{x}^{(0)}:=1,\ \mathbf{x}^{(1)}:=1-q+(\theta-1)(\mathbf{y}^{(1)}-1)$. Then, we have that $\mathbf{Fix}\{f_{\theta, q, 3}\}=\{\mathbf{x}^{(0)},\ \mathbf{x}^{(1)}\}$.
\end{theorem}

The proof follows from Proposition \ref{structureroot}.

\subsection{The Local Behavior of The Fixed Points}

We study the local behavior of the fixed points of the Potts--Bethe mapping \eqref{IsingPottsk=3}. In what follows, we assume $|\theta-1|_p<|q|_p<1$. 

Let $\lambda=\frac{d}{dx}f_{\theta,q,3}(\mathbf{x})$ where $\mathbf{x}$ is a fixed point of the Potts--Bethe mapping. Recall (see \cite{AVKhA2009,KhN2004}) that a fixed point $\mathbf{x}$ is called \textit{attracting} if $0 \leq |\lambda|_p < 1$, \textit{indifferent} if $|\lambda|_p=1$ and \textit{repelling} if $|\lambda|_p>1$.

\begin{theorem}\label{behaviourfixedpoints}
Let $p\geq 5$ with $ p \equiv 2 \ (mod \ 3)$  and $|\theta-1|_p<|q|_p<1,\ (\theta-1)(1-\theta-q)\neq 0.$  
The following statements are true:
\begin{itemize}
\item[(i)] $\mathbf{x}^{(0)}$ is an attracting fixed point;
\item[(ii)] $\mathbf{x}^{(1)}$ is a repelling fixed point.
\end{itemize}
\end{theorem}

\begin{proof}
We have to show that $\left|f'_{\theta, q, 3}(\mathbf{x}^{(0)})\right|_p<1$ and $\left|f'_{\theta, q, 3}(\mathbf{x}^{(1)})\right|_p>1$. It is easy to check that
\begin{eqnarray}\label{derivative}
f'_{\theta, q, 3}(\mathbf{x}^{(i)})&=&\frac{3(\theta-1)(\theta-1+q)(\theta \mathbf{x}^{(i)}+q-1)^2}{(\mathbf{x}^{(i)}+\theta+q-2)^4}\nonumber\\
&=&\frac{3(\theta-1)(\theta-1+q)\mathbf{x}^{(i)}}{(\theta \mathbf{x}^{(i)}+q-1)(\mathbf{x}^{(i)}+\theta+q-2)}
\end{eqnarray}
for $i=0,1.$ Hence, $\left|f'_{\theta, q, 3}(\mathbf{x}^{(0)})\right|_p=\frac{|\theta-1|_p}{|q|_p} < 1$.
It follows from \eqref{derivative} and $\mathbf{x}^{(1)}=1-q+(\theta-1)(\mathbf{y}^{(1)}-1)$  that
\begin{equation}\label{derivativefixed2}
f'_{\theta, q, 3}(\mathbf{x}^{(1)})=\frac{3(\theta-1+q)\mathbf{x}^{(1)}}{(\theta-1)\mathbf{y}^{(1)}(\theta \mathbf{y}^{(1)}+1-\theta-q)}.
\end{equation}
Due to Proposition \ref{structureroot}, we have that that $\left|\mathbf{y}^{(1)}\right|_p=|\mathbf{x}^{(1)}|_p=1$. Therefore, since $|\theta-1|_p<|q|_p<1,$ we get that $\left|f'_{\theta, q, 3}(\mathbf{x}^{(1)})\right|_p=\frac{|q|_p}{|\theta-1|_p} > 1$. This completes the proof.
\end{proof}

\subsection{The Basin of Attraction of The Attracting Fixed Point}
We describe a basin of attraction 
$$\mathfrak{B}(\mathbf{x}^{(0)}):=\left\{ x \in \mathbb{Q}_p : \lim_{n \to +\infty}f^{(n)}_{\theta, q, 3}(x)=\mathbf{x}^{(0)}\right\}$$
of the attracting fixed point $\mathbf{x}^{(0)}:=1$ where $f_{\theta, q, 3}^{(n+1)}(x)=f_{\theta, q, 3}\left(f_{\theta, q, 3}^{(n)}(x)\right)$ for $n\in\mathbb{N}$. 

We assume that $p\geq 5, \ |\theta-1|_p<|q|_p<1,$ and $p \equiv 2 \ (mod \ 3)$. 

We introduce the following sets
\begin{eqnarray*}
	\mathcal{A}_0 &:=& \left\{ x \in \mathbb{Q}_p : \left|x-\mathbf{x}^{(0)}\right|_p < |q|_p \right\}, \\ 
	\mathcal{A}_1 &:=& \left\{ x \in \mathbb{Q}_p : \left|x-\mathbf{x}^{(0)}\right|_p > |q|_p \right\}, \\ 
	\mathcal{A}_{0,\infty} &:=& \left\{ x \in \mathbb{Q}_p : \left|x-\mathbf{x}^{(\infty)}\right|_p = \left|x-\mathbf{x}^{(0)}\right|_p = |q|_p \right\}, \\ 
	\mathcal{A}_2 &:=& \left\{ x \in \mathbb{Q}_p : \left|\theta-1\right|_p < \left|x-\mathbf{x}^{(\infty)}\right|_p < |q|_p \right\},\\
	\mathcal{A}^{(1)}_{1,\infty} &:=& \left\{ x \in \mathbb{Q}_p : \left|x-\mathbf{x}^{(1)}\right|_p = \left|x-\mathbf{x}^{(\infty)}\right|_p = |\theta-1|_p \right\}, \\
	\mathcal{A}^{(2)}_{1,\infty} &:=& \left\{ x \in \mathbb{Q}_p : \left|x-\mathbf{x}^{(1)}\right|_p < \left|x-\mathbf{x}^{(\infty)}\right|_p = |\theta-1|_p \right\}, \\  
	\mathcal{A}_{\infty} &:=& \left\{ x \in \mathbb{Q}_p : \ 0<\left|x-\mathbf{x}^{(\infty)}\right|_p < \left|\theta-1\right|_p \right\}.
\end{eqnarray*}
It is clear that $\textup{\textbf{Dom}}\{f_{\theta,q,3}\}:=\mathbb{Q}_p\setminus \{\mathbf{x}^{\infty}\}=\mathcal{A}_0 \cup \mathcal{A}_1\cup \mathcal{A}_2\cup \mathcal{A}_{0,\infty} \cup \mathcal{A}^{(1)}_{1,\infty}\cup \mathcal{A}^{(2)}_{1,\infty}\cup\mathcal{A}_{\infty}.$

\begin{proposition}\label{p=2mod3}
	Let $p\geq 5$ with $p \equiv 2 \ (mod \ 3)$ and  $|\theta-1|_p < |q|_p < 1, \ (\theta-1)(1-\theta-q)\neq 0$. Then the following inclusions hold:
	\begin{itemize}
		\item[(i)] $\mathcal{A}_0 \cup \mathcal{A}_1 \cup \mathcal{A}_2 \cup \mathcal{A}_{0,\infty}\subset f^{(-1)}_{\theta,q,3}\left(\mathcal{A}_0\right);$
		\item[(ii)] $\mathcal{A}_{\infty}\subset f^{(-1)}_{\theta,q,3}\left(\mathcal{A}_{1} \right)\subset f^{(-2)}_{\theta,q,3}\left(\mathcal{A}_0\right);$
		\item[(iii)] $\mathcal{A}^{(1)}_{1,\infty}\subset f^{(-1)}_{\theta,q,3}\left(\mathcal{A}_{0,\infty} \right)\subset f^{(-2)}_{\theta,q,3}\left(\mathcal{A}_0\right).$
	\end{itemize}
\end{proposition}
\begin{proof}
(i) It is clear that
\begin{equation}\label{fminusx0p=2mod3}
f_{\theta,q,3}(x)-\mathbf{x}^{(0)} =\frac{(\theta-1)(x-\mathbf{x}^{(0)})}{(x-\mathbf{x}^{(\infty)})^3}g(x)
\end{equation}
where $g(x):=(\theta x+q-1)^2+(\theta x+q-1)(x-\mathbf{x}^{(\infty)})+(x-\mathbf{x}^{(\infty)})^2$. 

Let us first show that $f_{\theta,q,3}(x) \in \mathcal{A}_0$ for any $x \in \mathcal{A}_0 \cup \mathcal{A}_1$.
In this case, we have
$$
|x-\mathbf{x}^{(\infty)}|_p=\left|x-\mathbf{x}^{(0)}+q+(\theta-1)\right|_p=
\begin{cases}
|q|_p, & x \in \mathcal{A}_0, \\
|x-\mathbf{x}^{(0)}|_p, & x \in \mathcal{A}_1, \\
\end{cases}
$$
$$
|\theta x+q-1|_p=\left|\theta(x-\mathbf{x}^{(0)})+q+(\theta-1)\right|_p=
\begin{cases}
|q|_p, & x \in \mathcal{A}_0, \\
|x-\mathbf{x}^{(0)}|_p, & x \in \mathcal{A}_1, \\
\end{cases}
$$
$$|g(x)|_p \leq
\begin{cases}
|q|_p^2, & x \in \mathcal{A}_0, \\
|x-\mathbf{x}^{(0)}|_p^2, & x \in \mathcal{A}_1. \\
\end{cases}
$$
Therefore, we obtain from \eqref{fminusx0p=2mod3} that 
$$
|f_{\theta,q,3}(x)-\mathbf{x}^{(0)}|_p \leq
\left\{
\begin{array}{cc}
\frac{|x-\mathbf{x}^{(0)}|_p}{|q|_p}|\theta-1|_p, & x \in \mathcal{A}_0 \\
|\theta-1|_p, & x \in 
\mathcal{A}_1
\end{array}
\right\}
\leq |\theta-1|_p<|q|_p.
$$
This means that $f_{\theta,q,3}(x) \in \mathcal{A}_0$ for any $x \in \mathcal{A}_0 \cup \mathcal{A}_1$.

Let us now show that $f_{\theta,q,3}(x) \in \mathcal{A}_0$ for any $x \in \mathcal{A}_{0,\infty}$. Indeed, we get that
\begin{eqnarray*}
	|x-\mathbf{x}^{(0)}|_p=|x-\mathbf{x}^{(\infty)}|_p&=&|q|_p,\\	
	|\theta x+q-1|_p&=&\left|\theta(x-\mathbf{x}^{(\infty)})+(\theta-1)(1-\theta-q)\right|_p=|q|_p,\\
	|g(x)|_p&\leq& |q|^2_p. 
\end{eqnarray*}
We obtain from \eqref{fminusx0p=2mod3} that $|f_{\theta,q,3}(x)-\mathbf{x}^{(0)}|_p \leq {|\theta-1|_p}<|q|_p$ for any $x \in \mathcal{A}_{0,\infty}$.

It is easy to check that
\begin{eqnarray}\label{IsingPottsSingularp=2mod3}
f_{\theta, q, 3}(x)= \left(\theta+\frac{(\theta-1)(1-\theta-q)}{x-\mathbf{x}^{(\infty)}}\right)^3
\end{eqnarray}
and
\begin{eqnarray}\label{IsingPottsSingular2p=2mod3}
f_{\theta, q, 3}(x)-\mathbf{x}^{(0)}=\frac{(\theta-1)(x-\mathbf{x}^{(0)})}{x-\mathbf{x}^{(\infty)}}g_1(x)
\end{eqnarray}
where $g_1(x)=\left(\theta+\frac{(\theta-1)(1-\theta-q)}{x-\mathbf{x}^{(\infty)}}\right)^2+\left(\theta+\frac{(\theta-1)(1-\theta-q)}{x-\mathbf{x}^{(\infty)}}\right)+1$. 

Let us now show that $f_{\theta,q,3}(x) \in \mathcal{A}_0$ for any $x \in \mathcal{A}_2.$  Indeed, for any $x \in \mathcal{A}_2,$ we get that $\left|x-\mathbf{x}^{(0)}\right|_p=|q|_p$ and
\begin{eqnarray*}
	\left|\frac{(\theta-1)(1-\theta-q)}{x-\mathbf{x}^{(\infty)}}\right|_p&=&\frac{|q|_p|\theta-1|_p}{|x-\mathbf{x}^{(\infty)}|_p}<1,\\
	|g_1(x)|_p&=& 1,\\
	\left|f_{\theta, q, 3}(x)-\mathbf{x}^{(0)}\right|_p&=&\frac{\left|\theta-1\right|_p}{\left|x-\mathbf{x}^{(\infty)}\right|_p}\left|x-\mathbf{x}^{(0)}\right|_p<\left|x-\mathbf{x}^{(0)}\right|_p=|q|_p.
\end{eqnarray*}

Consequently, we show that $\mathcal{A}_0 \cup \mathcal{A}_1 \cup \mathcal{A}_2 \cup \mathcal{A}_{0,\infty}\subset f^{(-1)}_{\theta,q,3}\left(\mathcal{A}_0\right)$.

(ii) Now, we want to show that $f_{\theta,q,3}(x) \in \mathcal{A}_1$ for any $x \in \mathcal{A}_{\infty}.$ 

We consider the following sets
\begin{eqnarray*}
	\mathcal{A}^{(1)}_{\infty} &:=& \left\{ x \in \mathbb{Q}_p : \ 0<\left|x-\mathbf{x}^{(\infty)}\right|_p < |q|_p\left|\theta-1\right|_p \right\}\\
	\mathcal{A}^{(2)}_{\infty} &:=& \left\{ x \in \mathbb{Q}_p : \ \left|x-\mathbf{x}^{(\infty)}\right|_p = |q|_p\left|\theta-1\right|_p \right\}\\
	\mathcal{A}^{(3)}_{\infty} &:=& \left\{ x \in \mathbb{Q}_p : \ |q|_p\left|\theta-1\right|_p <\left|x-\mathbf{x}^{(\infty)}\right|_p < \left|\theta-1\right|_p \right\}
\end{eqnarray*}
where $\mathcal{A}_{\infty}=\mathcal{A}^{(1)}_{\infty}\cup \mathcal{A}^{(2)}_{\infty}\cup \mathcal{A}^{(3)}_{\infty}.$

Let $x \in \mathcal{A}^{(1)}_{\infty}$. It follows from \eqref{IsingPottsSingularp=2mod3} that 
\begin{eqnarray*}
	\left|\frac{(\theta-1)(1-\theta-q)}{x-\mathbf{x}^{(\infty)}}\right|_p&=&\frac{|q|_p|\theta-1|_p}{|x-\mathbf{x}^{(\infty)}|_p}>1,\\
	\left|f_{\theta,q,3}(x)\right|_p&=&\left|\theta+\frac{(\theta-1)(1-\theta-q)}{x-\mathbf{x}^{(\infty)}}\right|_p^3 > 1,\\
	\left|f_{\theta,q,3}(x)-\mathbf{x}^{(0)}\right|_p&=&\left|f_{\theta,q,3}(x)\right|_p>1>|q|_p.
\end{eqnarray*}
This means that  $f_{\theta,q,3}\left(\mathcal{A}^{(1)}_{\infty}\right)\subset\mathcal{A}_1$. 

Let $x \in \mathcal{A}^{(3)}_\infty$. Then $\left|x-\mathbf{x}^{(0)}\right|_p=|q|_p.$ It follows from \eqref{IsingPottsSingular2p=2mod3} that  
\begin{eqnarray*}
	\left|\frac{(\theta-1)(1-\theta-q)}{x-\mathbf{x}^{(\infty)}}\right|_p&=&\frac{|\theta-1|_p|q|_p}{|x-\mathbf{x}^{(\infty)}|_p}<1,\\ 
	\left|g_1(x)\right|_p&=& 1,\\
	|q|_p<\left|f_{\theta,q,3}(x)-\mathbf{x}^{(0)}\right|_p&=&\frac{\left|\theta-1\right|_p\left|x-\mathbf{x}^{(0)}\right|_p}{\left|x-\mathbf{x}^{(\infty)}\right|_p}<1.
\end{eqnarray*}
This means that $f_{\theta,q,3}\left(\mathcal{A}^{(3)}_\infty\right)\subset\mathcal{A}_1$. 

Let $x \in \mathcal{A}^{(2)}_{\infty}$ and $y:=\frac{x-\mathbf{x}^{(\infty)}}{\theta -1}.$ In this case, $|y|_p=|q|_p$ and $y=\frac{y^{*}}{|q|_p}$ where $y^{*}\in\mathbb{Z}^{*}_p.$ We then obtain that
\begin{eqnarray*}
	g(x)&=&(\theta x+q-1)^2+(\theta x+q-1)(x-\mathbf{x}^{(\infty)})+(x-\mathbf{x}^{(\infty)})^2\\
	&=&(\theta-1)^2\left[(\theta^2+\theta+1)y^2+(2\theta+1)(1-\theta-q)y+(1-\theta-q)^2\right]\\
	&=&\frac{(\theta-1)^2}{|q|^2_p}\left[(\theta^2+\theta+1)(y^{*})^2+(2\theta+1)(1-\theta-q)^{*}(y^{*})+\left((1-\theta-q)^{*}\right)^2\right]
\end{eqnarray*}

Now, we want to show that 
$$
\left|(\theta^2+\theta+1)(y^{*})^2+(2\theta+1)(1-\theta-q)^{*}(y^{*})+\left((1-\theta-q)^{*}\right)^2\right|_p=1
$$
for any $y^{*}\in\mathbb{Z}^{*}_p$ and $p\equiv 2 \ (mod \ 3).$ To do so, it is enough to show that the following congruent equation
$$
3t^2+3(1-\theta-q)^{*}t+\left((1-\theta-q)^{*}\right)^2\equiv 0 \ (mod \  p)
$$
 does not have any root in $\mathbb{Z}.$ Indeed, the discriminant $D=-3\left((1-\theta-q)^{*}\right)^2$ of the last quadratic congruent equation is \underline{a quadratic non-residue} whenever $p\equiv 2 \ (mod \ 3)$ (or equivalently $-3$ is \underline{a quadratic non-residue} whenever $p\equiv 2 \ (mod \ 3)$). Therefore, the last quadratic congruent equation  does not have any root in $\mathbb{Z}.$ 
 
Hence, we get that $|g(x)|_p=|q|^2_p|\theta-1|^2_p$ for any $x \in \mathcal{A}^{(2)}_{\infty}.$ It follows from \eqref{fminusx0p=2mod3} and $\left|x-\mathbf{x}^{(0)}\right|_p=|q|_p$ that $|f_{\theta,q,3}(x)-\mathbf{x}^{(0)}|_p=1>|q|_p$ for any $x\in \mathcal{A}^{(2)}_{\infty}.$ This means that $f_{\theta,q,3}\left(\mathcal{A}^{(2)}_{\infty}\right)\subset\mathcal{A}_1.$

Consequently, we show that $\mathcal{A}_{\infty}\subset f^{(-1)}_{\theta,q,3}\left(\mathcal{A}_{1} \right)\subset f^{(-2)}_{\theta,q,3}\left(\mathcal{A}_0\right).$

(iii) Let us show that $f_{\theta,q,3}(x) \in \mathcal{A}_{0,\infty}$ for any $x \in \mathcal{A}^{(1)}_{1,\infty}.$ 

We first show that $|f_{\theta,q,3}(x)-\mathbf{x}^{(0)}|_p=|q|_p$ for any  $x \in \mathcal{A}^{(1)}_{1,\infty}.$ Indeed, we obtain from \eqref{IsingPottsSingular2p=2mod3} that
\begin{eqnarray*}
	\left|\frac{(\theta-1)(1-\theta-q)}{x-\mathbf{x}^{(\infty)}}\right|_p&=&\frac{|\theta-1|_p|q|_p}{|x-\mathbf{x}^{(\infty)}|_p}<1,\\ 
	\left|g_1(x)-3\right|_p<\left|g_1(x)\right|_p&=& 1,\\
	\left|f_{\theta,q,3}(x)-\mathbf{x}^{(0)}\right|_p&=&\frac{\left|\theta-1\right|_p\left|x-\mathbf{x}^{(0)}\right|_p}{\left|x-\mathbf{x}^{(\infty)}\right|_p}=\left|x-\mathbf{x}^{(0)}\right|_p=|q|_p.
\end{eqnarray*}

We now show that $|f_{\theta,q,3}(x)-\mathbf{x}^{(\infty)}|_p=|q|_p$ for any  $x \in \mathcal{A}^{(1)}_{1,\infty}.$ Since  $|f_{\theta,q,3}(x)-\mathbf{x}^{(\infty)}|_p=|f_{\theta,q,3}(x)-\mathbf{x}^{(0)}+q+(\theta-1)|_p,$ it is enough to show that $|f_{\theta,q,3}(x)-\mathbf{x}^{(0)}+q|_p=|q|_p.$ It is easy to check that
\begin{multline*}
	f_{\theta,q,3}(x)-\mathbf{x}^{(0)}+q=\frac{(\theta-1)(x-\mathbf{x}^{(0)})}{x-\mathbf{x}^{(\infty)}}\left(g_1(x)-3\right)\\+3\left((\theta-1)-\frac{(\theta-1)^2}{x-\mathbf{x}^{(\infty)}}\right)+\frac{q(\theta-1)(\mathbf{y}^{(1)}-3)}{x-\mathbf{x}^{(\infty)}}+q\frac{x-\mathbf{x}^{(1)}}{x-\mathbf{x}^{(\infty)}}.
\end{multline*}
Since $|\mathbf{y}^{(1)}-3|_p=|x-\mathbf{x}^{(0)}|_p=|q|_p,\ |x-\mathbf{x}^{(\infty)}|_p=|x-\mathbf{x}^{(1)}|_p=|\theta-1|_p, \ |g_1(x)-3|_p<1,$ we obtain from the above equality that $|f_{\theta,q,3}(x)-\mathbf{x}^{(0)}+q|_p=|q|_p.$ This means that $|f_{\theta,q,3}(x)-\mathbf{x}^{(\infty)}|_p=|q|_p$.  Consequently, $f_{\theta,q,3}\left(\mathcal{A}^{(1)}_{1,\infty}\right)\subset\mathcal{A}_{0,\infty}$ or equivalently $\mathcal{A}^{(1)}_{1,\infty}\subset f^{(-1)}_{\theta,q,3}\left(\mathcal{A}_{0,\infty} \right)\subset f^{(-2)}_{\theta,q,3}\left(\mathcal{A}_0\right).$ This completes the proof.
\end{proof}

\begin{proposition}\label{setA21infty}
Let $p\geq 5$ with $p \equiv 2 \ (mod \ 3)$ and  $|\theta-1|_p < |q|_p < 1, \ (\theta-1)(1-\theta-q)\neq 0$. The following statements hold:
 \begin{itemize}
 	\item[(i)] For any $\bar{x},\bar{\bar{x}}\in\mathcal{A}^{(2)}_{1,\infty},$ one has that $\left|f_{\theta, q, 3}(\bar{x}) -f_{\theta, q, 3}(\bar{\bar{x}})\right|_p=\frac{|q|_p}{|\theta-1|_p}|\bar{x}-\bar{\bar{x}}|_p;$
 	\item[(ii)] For any $x\in\mathcal{A}^{(2)}_{1,\infty}\setminus\{\mathbf{x}^{(1)}\},$ there exists $n_0\in\mathbb{N}$ such that $f^{(n_0)}_{\theta,q,3}(x)\not\in\mathcal{A}^{(2)}_{1,\infty}.$
 \end{itemize}
\end{proposition}
\begin{proof}
(i) It is easy to check that 
\begin{equation*}
	f_{\theta, q, 3}(\bar{x}) -f_{\theta, q, 3}(\bar{\bar{x}})=
	\frac{(\theta-1)(1-\theta-q)(\bar{\bar{x}}-\bar{x})}{(\bar{x}-\mathbf{x}^{\infty})(\bar{\bar{x}}-\mathbf{x}^{\infty})}F_{\theta, q, 3}\left(\bar{x},\bar{\bar{x}}\right)
\end{equation*}
where
\begin{multline*}
	F_{\theta, q, 3}(\bar{x},\bar{\bar{x}})=3\theta^2+3\theta(1-\theta-q)\left(\frac{\theta-1}{\bar{x}-\mathbf{x}^{\infty}}+\frac{\theta-1}{\bar{\bar{x}}-\mathbf{x}^{\infty}}\right)\\
	+(1-\theta-q)^2\left(\frac{(\theta-1)^2}{\left(\bar{x}-\mathbf{x}^{\infty}\right)^2}+\frac{(\theta-1)^2}{\left(\bar{x}-\mathbf{x}^{\infty}\right)\left(\bar{\bar{x}}-\mathbf{x}^{\infty}\right)}+\frac{(\theta-1)^2}{\left(\bar{\bar{x}}-\mathbf{x}^{\infty}\right)^2}\right).
\end{multline*}

Let $\bar{x},\bar{\bar{x}}\in\mathcal{A}^{(2)}_{1,\infty}.$ We then have that
\begin{eqnarray*}
	\left|3\theta(1-\theta-q)\left(\frac{\theta-1}{\bar{x}-\mathbf{x}^{\infty}}+\frac{\theta-1}{\bar{\bar{x}}-\mathbf{x}^{\infty}}\right)\right|_p\leq |q|_p,\\
	\left|(1-\theta-q)^2\left(\frac{(\theta-1)^2}{\left(\bar{x}-\mathbf{x}^{\infty}\right)^2}+\frac{(\theta-1)^2}{\left(\bar{x}-\mathbf{x}^{\infty}\right)\left(\bar{\bar{x}}-\mathbf{x}^{\infty}\right)}+\frac{(\theta-1)^2}{\left(\bar{\bar{x}}-\mathbf{x}^{\infty}\right)^2}\right) \right|_p\leq|q|_p^2,\\
	|F_{\theta, q, 3}(\bar{x},\bar{\bar{x}})|_p=1\\
	\left|f_{\theta, q, 3}(\bar{x}) -f_{\theta, q, 3}(\bar{\bar{x}})\right|_p=\frac{|q|_p}{|\theta-1|_p}|\bar{x}-\bar{\bar{x}}|_p.
\end{eqnarray*}

(ii) Let $x\in\mathcal{A}^{(2)}_{1,\infty}\setminus\{\mathbf{x}^{(1)}\}$ be any point and $r:=|x-\mathbf{x}^{(1)}|_p>0.$ Since $\frac{|\theta-1|_p}{|q|_p}<1,$ there exists $n_0$ such that $\left(\frac{|\theta-1|_p}{|q|_p}\right)^{n_0}\leq\frac{r}{|\theta-1|_p}<\left(\frac{|\theta-1|_p}{|q|_p}\right)^{n_0-1}$ or equivalently $\left(\frac{|q|_p}{|\theta-1|_p}\right)^{n_0-1}r < |\theta-1|_p\leq \left(\frac{|q|_p}{|\theta-1|_p}\right)^{n_0}r.$ This means that 
$$
|f_{\theta, q, 3}^{(n_0-1)}(x)-\mathbf{x}^{(1)}|_p=\frac{|q|^{n_0-1}_pr}{|\theta-1|^{n_0-1}_p} < |\theta-1|_p\leq \frac{|q|^{n_0}_pr}{|\theta-1|^{n_0}_p}=|f_{\theta, q, 3}^{(n_0)}(x)-\mathbf{x}^{(1)}|_p
$$
or equivalently
$$
f_{\theta, q, 3}^{(n_0-1)}(x)\in \mathcal{A}^{(2)}_{1,\infty} \quad  \textup{and} \quad  f_{\theta, q, 3}^{(n_0)}(x)\not\in \mathcal{A}^{(2)}_{1,\infty}.
$$
This completes the proof.
\end{proof}

We now describe the basin of attraction of the attracting fixed point.

\begin{theorem}
Let $p\geq 5$ with $p \equiv 2 \ (mod \ 3)$ and  $|\theta-1|_p < |q|_p < 1, \ (\theta-1)(1-\theta-q)\neq 0$. Then, we have that $$\mathfrak{B}\left( \mathbf{x}^{(0)} \right)= \mathbb{Q}_p\setminus\left( \{\mathbf{x}^{(1)}\}\cup \bigcup\limits_{n=0}^{+\infty}f_{\theta, q, 3}^{(-n)}\{\mathbf{x}^{(\infty)}\}\right).$$ 
\end{theorem}

\begin{proof}
Let  $\bigcup\limits_{n=0}^{+\infty}f_{\theta, q, 3}^{(-n)}\{\mathbf{x}^{(\infty)}\}$ be a set of all eventually singular points of the Potts--Bethe mapping.  Let $x\in\mathcal{A}_0.$ It follows from Proposition \ref{p=2mod3} (i) that $|f_{\theta, q, 3}(x)-\mathbf{x}^{(0)}|_p \leq \frac{|\theta-1|_p}{|q|_p}|x-\mathbf{x}^{(0)}|_p$ where $\frac{|\theta-1|_p}{|q|_p}<1$. 
This means that $f_{\theta, q, 3}\left(\mathcal{A}_0\right)\subset \mathcal{A}_0$. Let $x\in \mathfrak{B}\left(\mathbf{x}^{(0)}\right)$ be any point. This means that $\lim\limits_{n \to +\infty}f^{(n)}_{\theta, q, 3}(x)=\mathbf{x}^{(0)}.$ Then there exist some $n_0\in\mathbb{N}$ such that $f_{\theta, q, 3}^{(n_0)}(x)\in\mathcal{A}_0.$ This means that $x\in f_{\theta,q,3}^{(-n_0)}\left(\mathcal{A}_0\right)$ or equivalently
$$\mathfrak{B}\left(\mathbf{x}^{(0)} \right) = \bigcup^{+\infty}_{n = 0}f_{\theta,q,3}^{(-n)}\left(\mathcal{A}_0\right).$$
Due to Proposition  \ref{p=2mod3} (i)-(iii), we have that 
$$
\textup{\textbf{Dom}}\{f_{\theta,q,3}\}\setminus \mathcal{A}^{(2)}_{1,\infty} \subset f_{\theta,q,3}^{(-1)}\left(\mathcal{A}_0\right)\cup f_{\theta,q,3}^{(-2)}\left(\mathcal{A}_0\right)\subset \mathfrak{B}\left(\mathbf{x}^{(0)} \right). 
$$
Moreover, due to Proposition \ref{setA21infty} (ii), for any $x\in \mathcal{A}^{(2)}_{1,\infty}\setminus \{\mathbf{x}^{(1)}\}$ there exists $n_0$ such that $f^{(n_0)}_{\theta,q,3}(x)\in \mathbb{Q}_p\setminus \mathcal{A}^{(2)}_{1,\infty}\subset\mathfrak{B}\left( \mathbf{x}^{(0)} \right)\cup \{\mathbf{x}^{(\infty)}\}.$ This completes the proof.
\end{proof}			


\section{Conclusions}

In this paper, we have studied the dynamics of the Potts--Bethe mapping associated with the  $p-$adic $q-$state Potts model on the Cayley tree of order-3 in the $p-$adic field $\mathbb{Q}_p$ with $p \equiv 2 \ (mod \ 3)$.  Namely, we have showed that the Potts--Bethe mapping has two fixed points one of them is attracting and another one is repelling. Moreover, a trajectory of the Potts--Bethe mapping starting from an initial eventually non-singular point converges to the attracting point.

\section*{Acknowledgments}
The first author (M.S.) is grateful to the Junior Associate scheme of the Abdus Salam International Centre for Theoretical Physics, Trieste, Italy.  The second author (M.A.Kh.A) is grateful to the MOHE grant FRGS14-141-0382 for the financial support. He is also indebted  to  Embassy of France in Malaysia and Labex B\'ezout, Universit\'e Paris--Est for the financial support to pursue his PhD study at LAMA, Universit\'e Paris--Est Cr\'eteil, France.


\section*{References}
\begin{thebibliography}{9}


	\bibitem{AVKhA2009} Anashin V and Khrennikov A Yu 2009 \textit{Applied algebraic dynamics} (Walter de Gruyter)
	
	\bibitem{Arndt} Arndt P F 2000 \textit{Phys. Rev. Let.} \textbf{84} 814--817
	
	\bibitem{AHR}  Arndt P F,  Heinzel T and  Rittenberg V 1999 \textit{J. Stat. Phys.} \textbf{97}(1) 1--65
	
	\bibitem{Binek} Binek Ch 1998 \textit{Phys. Rev. Let.} \textbf{81} 5644--5647 
	
	\bibitem{BoscoRosa1987} Bosco F A and  Goulart Rosa S 1987 \textit{EPL} \textbf{4} (10) 1103--1108
	
	\bibitem{BorShaf} Borevich Z I and Shafarevich I R 1966   \textit{Number Theory} (Acad. Press, New York)
	
	\bibitem{DDI1983} Derrida  B,  De Seze L and  Itzykson C 1983 \textit{J. Stat. Phys.} \textbf{33} 559--569
	
	\bibitem{Liao2} Fan A H, Fan S L, Liao L M and Wang Y F 2014 {\it Advances in Mathematics} \textbf{257} 92--135
	
	\bibitem{Fisher} Fisher M E, 1965 \textit{Lectures in Theoretical Physics} (edited by W. E. Brittin) \textbf{VII C},  1--159.
	
	\bibitem{GMR} Ganikhodjaev N N, Mukhamedov F M, and Rozikov U A 2002 \textit{Theo. Math. Phys.} {\bf 130}(3) 425--431
	
	\bibitem{IPZ1983} Itzykson C, Pearson R B and  Zuber J B 1983 \textit{Nuclear. Phys.} \textbf{B 220} 415--433.
	
    \bibitem{Kh1994} Khrennikov A Yu 1994  \textit{$p$-Adic Valued Distributions in Mathematical Physics} (Kluwer)
	
	\bibitem{Kh2009} Khrennikov A Yu 2009 \textit{Interpretations of Probability} (Walter de Gruyter, Berlin, New York)

	
	\bibitem{KhN2004} Khrennikov A Yu and Nilsson M 2004  \textit{p-adic deterministic and random dynamical systems} (Kluwer)
	
	\bibitem{Kh1999} Khrennikov A Yu, Yamada Sh and van Rooij A 1999 \textit{Annal Math Blaise Pascal} \textbf{6}(1) 21--32
	
	\bibitem{NK} Koblitz N 1984  \textit{$p-$adic numbers, $p-$adic Analysis, and Zeta Functions} (Springer New York)
	
	\bibitem{CKURRK} Kulske C, Rozikov U A and Khakimov R M 2013 \textit{J. Stat. Phys.} \textbf{156}(1) 189--200
	
	\bibitem{LeeYang} Lee T D and Yang C N 1952 \textit{Phys. Rev.} \textbf{87} 410.
	
	\bibitem{LuWu1998} Lu W T and  Wu F Y 1998 \textit{Physica A} \textbf{258} 157--170 
	
	\bibitem{LuWu2000} Lu W T and  Wu F Y 2000  \textit{J. Stat. Phys.} \textbf{102} 953-970 
	
	\bibitem{Monroe1991} Monroe J L 1991 \textit{J. Stat. Phys.} \textbf{65}(3-4) 445--452
	
	\bibitem{Monroe1992} Monroe J L 1992 \textit{J. Stat. Phys.} \textbf{67}(5-6) 1185--1200
	
	\bibitem{Monroe1994} Monroe J L 1994 \textit{Phys. Let. A} \textbf{188}(1) 80--84 
	
	\bibitem{Monroe1995} Monroe J L 1995 \textit{EPL} \textbf{29}(2) 187--188
	
	\bibitem{Monroe1996} Monroe J L 1996 \textit{J. Phys. A: Math. Gen.} \textbf{29} (17) 5421--5427
	
	\bibitem{Monroe2001} Monroe J L 2001 \textit{J. Phys. A: Math. Gen.} \textbf{34}(33) 6405--6412
	
	\bibitem{M2} Mukhamedov F 2013 \textit{Math. Phys. Anal. Geom.} {\bf 16} 49--87
	
	\bibitem{M3} Mukhamedov F 2015 \textit{International Journal of Theoretical Physics} {\bf 54} (10) 3577--3595
	
	\bibitem{MH2013} Mukhamedov F and Akin H 2013 \textit{J. Stat. Mechanics: Theory and Experiment} P07014
	
	\bibitem{FMBOMSKM} Mukhamedov F, Omirov B, Saburov M,  and Masutova K 2013 \textit{Sib Math Journal} \textbf{54} 501--516
	
	\bibitem{FMBOMS} Mukhamedov F, Omirov B, and Saburov M 2014 {\it Int. J. Number Theory} \textbf{10} 1171--1190
	
	\bibitem{MR1} Mukhamedov F and Rozikov U 2004 \textit{Indag. Math. (N.S.)} {\bf 15} 85--100
	
	\bibitem{MR2} Mukhamedov F and Rozikov U 2005 \textit{Infin. Dimens. Anal. Quantum Probab. Relat. Top.} \textbf{8}(2)
	277--290
	
	\bibitem{FMMS} Mukhamedov F and Saburov M 2013 {\it J. Number Theory} \textbf{133} (1) 55--58
	
	\bibitem{FMMSOK2015} Mukhamedov F, Saburov M and Khakimov O 2015 \textit{J. Stat. Mechanics: Theory and Experiment} P05032
	
	\bibitem{FMMSOK2016} Mukhamedov F, Saburov M and Khakimov O 2016 \textit{Theoretical and Mathematical Physics} \textbf{187} (1) 583--602
	
	\bibitem{FMOK} Mukhamedov F and Khakimov O 2016 \textit{Chaos Solitons and Fractals} \textbf{87} 190--196
	
	\bibitem{PF1999} Park Y and  Fisher M 1999 \textit{Phys Rev E}  \textbf{60} 6323--6328 
	
	\bibitem{RUBook} Rozikov U 2013 \textit{Gibbs measures on Cayley trees} (World Sci. Pub. Singapore)
	404 pp
	
	\bibitem{RUSurvey} Rozikov U 2013 \textit{Rev. Math. Phys.} \textbf{25}(1) 1330001 (112 pages)
	
	\bibitem{RUKO2013} Rozikov U and Khakimov O 2013 \textit{Theo. Math. Phys.} \textbf{175}(1) 518--525
	
	\bibitem{RUKO2015} Rozikov U and Khakimov O 2015 {\it Markov Processes and Relat. Fields}  {\bf 21} 177--204
	
	\bibitem{Ruelle} Ruelle D 1971 \textit{Phys. Rev. Let.} \textbf{26}, 303--304 
	
	\bibitem{SMAA2014} Saburov M and Ahmad M A Kh 2014 \textit{AIP Conference Proceedings} {\bf 1602} 792--797
	
	\bibitem{SMAA2015a} Saburov M and Ahmad M A Kh 2015 \textit{Sains Malaysiana} \textbf{44}(4) 635--641
	
	\bibitem{SMAA2015b} Saburov M and Ahmad M A Kh 2015 \textit{Sains Malaysiana}  \textbf{44}(5) 765--769
	
	\bibitem{SMAA2015c} Saburov M and Ahmad M A Kh 2015 \textit{ScienceAsia} \textbf{41}(3) 209--215
	
	\bibitem{SMAA2015d} Saburov M and Ahmad M A Kh 2015 \textit{Math Phys Anal Geom} \textbf{18} 1--33
	
	\bibitem{SMAA2016a} Saburov M and Ahmad M A Kh 2016 \textit{Malaysian Journal of Mathematical Sciences} \textbf{10(S)} 15--35
	
	\bibitem{SMAA2016b} Saburov M and Ahmad M A Kh 2016 \textit{Bulletin of the Malaysian Mathematical Sciences Society} (Accepted)
	
	\bibitem{YangLee} Yang C N  and  Lee T D 1952 Phys. Rev. \textbf{87} 404.
 
\end{thebibliography}
\end{document}